\documentclass[leqno,10pt]{amsart}
\usepackage{amsfonts, amsmath, amssymb, amsthm, color}
\allowdisplaybreaks[1]
\newtheorem{thm}{Theorem}[section]
\newtheorem{lemma}{Lemma}[section]
\newtheorem{prop}{Proposition}[section]

\newtheorem{rmk}{Remark}[section]
\theoremstyle{definition}
\numberwithin{equation}{section}
\newcommand{\rr}{\mathbb R}
\newcommand{\al}{\alpha}
\newcommand{\de}{\delta}

\newcommand{\la}{\lambda}
\newcommand{\ga}{\gamma}
\newcommand{\si}{\sigma}

\newcommand{\Om}{\Omega}

\newcommand{\pl}{\partial}
\newcommand{\calP}{\mathcal P}

\newcommand{\BB}{B_1}

\newcommand{\Bs}{B_{\sigma}}
\newcommand{\Ud}{U_\delta}
\newcommand{\bl}{\bar\la_{\tau,\ga}}
\newcommand{\Jd}{J_\la^{\mathrm{(d)}}}
\newcommand{\Js}{J_\la^{\mathrm{(s)}}}
\newcommand{\lstar}{\la_{\tau,\ga}^*}
\newcommand{\btg}{\beta_{\tau,\ga}}
\newcommand{\sitg}{\sigma_{\tau,\ga}}
\newcommand{\Latg}{\Lambda_{\tau,\ga}}
\newcommand{\uz}{\underline z}
\newcommand{\blPd}{\bar\lambda_{\mathcal P}^{(d)}}

\def\beq{\begin{equation}}
\def\beqq{\begin{equation*}}
\def\eeq{\end{equation}}
\def\eeqq{\end{eqnarray*}}
\def\bal{\begin{aligned}}
\def\eal{\end{aligned}}
\def\bca{\begin{cases}}
\def\eca{\end{cases}}

\def\sideremark#1{\ifvmode\leavevmode\fi\vadjust{\vbox to0pt{\vss
 \hbox to 0pt{\hskip\hsize\hskip1em
 \vbox{\hsize3cm\tiny\raggedright\pretolerance10000
  \noindent #1\hfill}\hss}\vbox to8pt{\vfil}\vss}}}%

\begin{document}
\title[Radial Onsager vortices]{On radial two-species Onsager vortices near the critical temperature}
\author[T.~Ricciardi]{Tonia Ricciardi${}^\ast$}\thanks{${}^\ast$Corresponding author}
\address[T.~Ricciardi] {Dipartimento di Matematica e Applicazioni,
Universit\`{a} di Napoli Federico II, Via Cintia, Monte S.~Angelo, 80126 Napoli, Italy}
\email{tonricci@unina.it}
\author[R.~Takahashi]{Ryo Takahashi}
\address[R.~Takahashi] {Mathematics Education, Faculty of Education, Nara University of Education,
Takabatake-cho, Nara-shi 630--8528, Japan}
\email{r-takaha@nara.edu-ac.jp}
\date{\today}
\begin{abstract}
We compare two mean field equations describing hydrodynamic 
turbulence in equilibrium, which are derived
under a deterministic vs.\ stochastic assumption
on the variable vortex intensity distribution. 
Mathematically, such equations correspond to 
non-local Liouville type problems, and
the critical temperature corresponds to the optimal Moser-Trudinger constant.
We consider the radial case and we assume that the inverse temperature is near its critical value. 
Under these assumptions we show that, unlike previously existing results, the qualitative properties of the
solution set in the deterministic case is more similar to the single vortex intensity case than
the stochastic case.
Some new variational interpretations of the 
value explicit values of the critical temperature are also provided.
\end{abstract}
\subjclass[2000]{}
\maketitle
\section{Introduction}
\label{sec:intro}
In recent years, in the context of the statistical mechanics of two-dimensional point vortices, as 
introduced in the pioneering work of Onsager~\cite{Onsager}, several
non-local elliptic equations including an exponential type nonlinearity
have been derived in order to take into account of variable vortex intensities.
See, e.g., \cite{Chavanis, Suzuki2008book} and the references therein.
\par
In particular, assuming that the distribution of the (normalized) vortex intensity $\al$
is determined by a probability distribution 
$\mathcal P(d\al)$, $\al\in[-1,1]$,
the following multi-species Boltzmann-Poisson type problem was derived in \cite{SawadaSuzuki},
see also \cite{EyinkSreenivasan, SopikSireChavanis2005}:
\begin{equation}
\label{eq:detP}
\begin{cases}
\begin{aligned}
-\Delta v=&\la\int_{[-1,1]}\al\frac{e^{\al v}}{\int_\Om e^{\al v}}\,\mathcal P(d\al)
&&\hbox{in\ }\Om\\
v=&0
&&\hbox{on\ }\pl\Om.
\end{aligned}
\end{cases}
\end{equation}
Here, $\Om\subset\rr^2$ is a smooth bounded domain, $v$ denotes the stream function of the turbulent Euler flow 
and $\la>0$ is a physical constant related to the inverse temperature.
It should be mentioned that problems of the form \eqref{eq:detP} 
also appear in the description of equilibrium states 
for Brownian gases and chemotaxis systems with different types of particle species, 
subject to the conserved mass constraint for each species, see 
\cite{Chavanis2007, Chavanis2012, Horstmann2011, KavallarisRicciardiZecca},
and the references therein, as well as in affine geometry, see
\cite{JevnikarYang}. 
Here, we are particularly interested in the \lq\lq two-species" case, where
$\calP$ is given by
\beq
\label{def:Ptg}
\mathcal P(d\al)=\tau\de_1(d\al)+(1-\tau)\de_{\ga}(d\al),
\eeq
for some $\tau,\ga\in(0,1)$. For such a choice of $\mathcal P$,
problem~\eqref{eq:detP} reduces to the problem
\begin{equation}
\label{eq:det}
\begin{cases}
\begin{aligned}
-\Delta v=&\la\left(\tau\frac{e^v}{\int_{\Om}e^v\,dx}
+(1-\tau)\ga\frac{e^{\ga v}}{\int_{\Om}e^{\ga v}\,dx}\right)
&&\hbox{in\ }\Om\\
v=&0
&&\hbox{on\ }\pl\Om.
\end{aligned}
\end{cases}
\end{equation}
\par
On the other hand, the following problem, formally similar to \eqref{eq:detP}, was derived in \cite{Neri} 
in order to describe stationary turbulent flows  
under a \lq\lq stochastic" assumption on the vorticity distribution:
\begin{equation}
\label{eq:stochP}
\left\{
\bal
-\Delta v=&\la\frac{\int_{[-1,1]}\al e^{\al v}\,\mathcal P(d\al)}
{\iint_{\Om\times[-1,1]}e^{\al v}\,dx\mathcal P(d\al)}
&&\hbox{in }\Om\\
v=&0&&\hbox{on }\pl\Om.
\eal
\right.
\end{equation}
In this case, it is assumed that the vortex intensities $\al\in[-1,1]$ of each point vortex 
are independent identically distributed random variables with
probability distribution $\mathcal P(d\al)$.
In the \lq\lq two-species case" where $\mathcal P$ is given by \eqref{def:Ptg},
problem~\eqref{eq:stochP} takes the form:
\begin{equation}
\label{eq:stoch}
\left\{
\bal
-\Delta v=&\la\frac{\tau e^v+(1-\tau)\ga e^{\ga v}}{\int_{\Om}\{\tau e^v+(1-\tau)e^{\ga v}\}\,dx}
&&\hbox{in }\Om\\
v=&0&&\hbox{on }\pl\Om.
\eal
\right.
\end{equation}
Problem~\eqref{eq:stochP} may also be viewed as the stationary state of
a multi-species Boltzmann-Poisson evolution system, where 
the masses of the individual population species
are allowed to vary, provided the total population mass is conserved, see \cite{KavallarisRicciardiZecca}.
\par
In the special case $\mathcal P(d\al)=\de_1(d\al)$, 
equation~\eqref{eq:detP} and equation~\eqref{eq:stochP} 
both reduce to the well-understood single-species Boltzmann-Poisson system,
also known in the mathematical literature as the as \lq\lq standard" mean field equation
\begin{equation}
\label{eq:std}
\left\{
\bal
\begin{aligned}
-\Delta v=&\la\frac{e^v}{\int_{\Om}e^v\,dx}
&&\hbox{in\ }\Om\\
v=&0
&&\hbox{on\ }\pl\Om,
\end{aligned}
\eal
\right.
\end{equation}
which has been extensively analyzed in view of its applications in biology, differential geometry
and physics.
In the context of turbulence, problem~\eqref{eq:std} was rigorously derived and analyzed in \cite{CLMP}.
See \cite{CSLinreview, Malchiodi, JostWangYeZhou, Suzuki2008book} and the references therein 
for results concerning problem~\eqref{eq:std}.
\par
In recent years, a considerable effort has been devoted to compare qualitative
properties of the solutions to problem~\eqref{eq:det} and problem~\eqref{eq:stoch},
see \cite{RiZe2016, RicciardiSuzuki, RicciardiTakahashi, 
JevnikarYang, DeMarchisRicciardi, ORS, RicciardiTakahashiZeccaZhang}
In particular, in \cite{RiZe2012, RiZe2016, DeMarchisRicciardi} it is shown that
problem~\eqref{eq:stoch} may be viewed as
a perturbation of \eqref{eq:std}, and in particular it shares the same mass quantization
properties and optimal Moser-Trudinger constant (which physically corresponds to the critical temperature) 
as \eqref{eq:std}, independently of $\mathcal P$, see also \cite{Suzuki2008book}. 
On the other hand, in \cite{ORS, RicciardiSuzuki, JevnikarYang, GuiJevnikarMoradifam}
it is shown that the blow-up masses and the optimal Moser-Trudinger constant 
for  \eqref{eq:det} significantly depend on $\mathcal P$, and that
\eqref{eq:det} is not in general a perturbation of \eqref{eq:std},
with the exception of the case where the probability measure $\mathcal P(d\al)$ is sufficiently \lq\lq close" 
to the Dirac mass $\de_1(d\al)$, 
see \cite{RicciardiTakahashiZeccaZhang}. 
Sign-changing cases were considered in \cite{JevnikarYang2017, PistoiaRicciardi2016, PistoiaRicciardi,
JostWangYeZhou, RicciardiTakahashi}.
\par
Our aim in this note is to exhibit a first situation where the deterministic
problem~\eqref{eq:det} behaves more similarly to the \lq\lq standard\rq\rq problem~\eqref{eq:std}
than the stochastic problem~\eqref{eq:stoch}.
More precisely, we shall show that when $\Omega$ is a disc and $\mathcal P$ is of the form~\eqref{def:Ptg},
the critical temperature for \eqref{eq:det} determines a threshold for existence of solutions,
exactly as it happens for \eqref{eq:std}, independently of $\tau,\ga$.
On the other hand, bifurcation diagram of the solution set near the critical temperature for \eqref{eq:stoch}
significantly depends on $\tau,\ga$, and does not necessarily determine a threshold for 
the existence of solutions.
\par
This note is organized as follows. 
In the following section, we state our main results, namely Theorem~\ref{thm:det} and Theorem~\ref{thm:Neri} below. 
Several preliminary properties necessary to the proofs are derived in Section~\ref{sec:prelims}. 
The proof of Theorem~\ref{thm:det} is provided in Section~\ref{sec:det}, and
the proof of Theorem~\ref{thm:Neri} is provded in Section~\ref{sec:stoch}. 
Finally, some remarks on the explicit value of the critical temperature 
for the deterministic problem~\eqref{eq:det}, namely the constant~$\bl$ 
deined in \eqref{eq:MTdet} below, are provided in the Appendix. 
\section{Statement of the main results}
\label{sec:results}
We first recall the variational formulations of 
problem~\eqref{eq:detP} and problem~\eqref{eq:stochP},
given by the functionals
\beq
\label{def:JdetP}
J_{\la,\mathcal P}^{(d)}(v)=\frac{1}{2}\int_\Om|\nabla v|^2
-\int_{[-1,1]}\left[\ln\int_\Om e^{\al v}\,dx\right]\mathcal P(d\al),
\qquad v\in H_0^1(\Om),
\eeq
and
\beq
\label{def:JstochP}
J_{\la,\mathcal P}^{(s)}(v)=\frac{1}{2}\int_\Om|\nabla v|^2
-\ln\left[\iint_{\Om\times[-1,1]} e^{\al v}\,dx\mathcal P(d\al)\right],
\qquad v\in H_0^1(\Om),
\eeq
respectively.
For $\calP(d\al)=\de_1(d\al)$, \eqref{def:JdetP} and \eqref{def:JstochP}
reduce to the variational functional for \eqref{eq:std},
given by 
\beq
\label{def:IMT}
I_\la(v)=\frac{1}{2}\int_\Om|\nabla v|^2\,dx-\la\ln\int_\Om e^v\,dx,
\qquad v\in H_0^1(\Om).
\eeq
In view of the Moser-Trudinger inequality \cite{Moser,Trudinger}
\[
\int_\Om e^{4\pi v^2/\int_\Om|\nabla v|^2\,dx}\le C_{MT}|\Om|
\qquad\hbox{for all\ }v\in H_0^1(\Om),
\]
where the constant $C_{MT}>0$ does not depend on the domain~$\Om$
and where the constant $4\pi$ is best possible, we derive that
\[
\label{ineq:MT}
\int_\Om e^{v}\,dx\le C_{MT}|\Om|e^{\frac{1}{16\pi}\int_\Om|\nabla v|^2\,dx}
\qquad\hbox{for all\ }v\in H_0^1(\Om).
\]
In particular, it follows that
\beq
\label{ineq:MT}
\inf_{H_0^1(\Om)}I_\la>-\infty\quad\hbox{if and only if }\la\le8\pi,
\eeq
and that for all $\la<8\pi$ a solution for \eqref{eq:std} may be obtained by
direct minimization of $I_\la$.
The limit value $\la=8\pi$ corresponds to the \lq\lq critical temperature"
for the single species case $\calP(d\al)=\de_1(d\al)$.
Such a solution is also unique \cite{Suzuki1992, BartolucciLin2014}. 
On the other hand, if $\Om$ is a disc, then a Pohozaev identity argument,
as may be found, e.g., in \cite{BebernesLacey}
(or, alternatively, the explicit form of the solutions, see \cite{Bandle}) implies 
that problem~\eqref{eq:std}
does not admit solutions when $\la\ge8\pi$.
Such existence/non-existence properties were employed in an essential way in \cite{KavallarisSuzuki}
to establish the biologically relevant
finite-time blow-up of solutions
to the corresponding evolution problem for \eqref{eq:std}. 
\par
Motivated by the recent results in \cite{KavallarisRicciardiZecca},
our aim in this note is to investigate similar properties for \eqref{eq:det} and \eqref{eq:stoch}.
In order to state our results precisely, we first recall the corresponding extensions
of the \lq\lq optimal Moser-Trudinger inequality"~\eqref{ineq:MT} to the functionals $J_{\la,\calP}^{(d)}$
and $J_{\la,\calP}^{(s)}$ as in \cite{RicciardiSuzuki, RiZe2012, Suzuki2008book}.
\par
Let $\blPd$ be defined by
\[
\blPd:=8\pi\inf\left\{\frac{\calP(K_\pm)}
{\left(\int_{K_\pm}\al\,\calP(d\al)\right)^2}:\ K_\pm\subset I_\pm,
\ \calP(K_\pm)>0
\right\},
\]
where $I_+:=[0,1]$ and $I_-:=[-1,0)$.
We say that $\calP$ is discrete if $\calP=\sum_{i=1}^m\tau_i\de_{\ga_i}$
for some $m\in\mathbb N$, $\tau_i,\in(0,1]$, $\sum_{i=1}^m\tau_i=1$, $\ga_i\in[-1,1]$.
The following optimal conditions for boundedness from below of $J_{\la,\calP}^{(d)}$
and $J_{\la,\calP}^{(s)}$ hold true.
\begin{prop}[\cite{RicciardiSuzuki, RiZe2012, ShafrirWolansky}]
\label{prop:MTP}
There holds:
\begin{enumerate}
  \item[(i)]
$\inf_{H_0^1(\Om)}J_{\la,\calP}^{(d)}>-\infty$ if $\la<\blPd$ and  
$\inf_{H_0^1(\Om)}J_{\la,\calP}^{(d)}=-\infty$ if $\la>\blPd$.
Moreover, if $\calP$ is discrete, then $\inf_{H_0^1(\Om)}J_{\la,\calP}^{(d)}\vert_{\la=\blPd}>-\infty$.
\item[(ii)]
There holds  
$\inf_{H_0^1(\Om)}J_{\la,\calP}^{(s)}>-\infty$ if and only if $\la\le 8\pi$.
\end{enumerate}
\end{prop}
In the special case where $\calP$ is of the two-species form~\eqref{def:Ptg},
the functional $J_{\la,\calP}^{(d)}$ takes the form
\beq
\label{def:Jdet}
\Jd(v)=\frac{1}{2}\int_\Om|\nabla v|^2\,dx-\la\tau\ln\int_\Om e^v\,dx
-\la(1-\tau)\ln\int_\Om e^{\ga v}\,dx,
\qquad v\in H_0^1(\Om),
\eeq
corresponding to the variational functional for \eqref{eq:det}.
In view of Proposition~\ref{prop:MTP}--(i),
the resulting Moser-Trudinger inequality for $\Jd$ is given by
\beq
\label{ineq:JdMT}
\inf_{H_0^1(\Om)}\Jd>-\infty\quad\hbox{if and only if }\la\le\bl
\eeq
where $\bl=\blPd\vert_{\calP=\tau\de_1+(1-\tau)\de_\ga}$ is given by:
\begin{equation}
\label{eq:MTdet}
\bl=8\pi\min\left\{\frac{1}{\{\tau+(1-\tau)\ga\}^2},\frac{1}{\tau}\right\}
=
\begin{cases}
\frac{8\pi}{\tau},
&\hbox{if }0<\ga\le\frac{\sqrt\tau}{1+\sqrt\tau}\\
\frac{8\pi}{\{\tau+(1-\tau)\ga\}^2},
&\hbox{if }\frac{\sqrt\tau}{1+\sqrt\tau}<\ga<1.
\end{cases}
\end{equation}
\begin{rmk}
Roughly speaking, 
the case $\bl=8\pi/\tau$ corresponds to the case where the term $e^{\ga v}$
is a perturbation of the term $e^v$, whereas the case $\bl=8\pi/\{\tau+(1-\tau)\ga\}^2$ corresponds 
to the situation where
both terms $e^v$, $e^{\ga v}$ influence the Moser-Trudinger constant. 
When $\ga\ge1/2$, straightforward considerations 
involving \lq\lq Liouville bubbles" imply that both nonlinearities should 
concentrate along a blowing-up sequence,
so that the first case is natural.
On the other hand, the threshold value satisfies 
$\sqrt{\tau}/(1+\sqrt{\tau})<1/2$ for all $\tau\in(0,1)$.
So, in the interval $(\sqrt{\tau}/(1+\sqrt{\tau}),1/2)$
it the functional is also influeced by the lower-order term, although it does not blow up along
the sequence of minimizers.
See the Appendix for some considerations which motivate the role of $\tau$
in this context.
\end{rmk}
In particular, for all $\la<\bl$ there exists a solution to \eqref{eq:det} corresponding to
the minimum of $\Jd$. Our first result states that if $\Om$ is the unit disc, similarly as for
the \lq\lq standard" case $\tau=1$, the optimal Moser-Trudinger constant $\bl$ provides
a \emph{sharp threshold} for the existence of solutions to \eqref{eq:det}:
\begin{thm}
\label{thm:det}
Suppose $\Om=\BB(0)$.
For any $\tau,\ga\in(0,1)$, a solution to \eqref{eq:det}
exists if and only if $\la<\bl$.
\end{thm}
On the other hand, the variational functional for 
\eqref{eq:stoch} is given by
\beq
\label{def:Jstoch}
\Js(v)=\frac{1}{2}\int_\Om|\nabla v|^2\,dx
-\la\ln\left(\tau\int_\Om e^v\,dx+(1-\tau)\int_\Om e^{\ga v}\,dx\right),
\qquad v\in H_0^1(\Om),
\eeq
corresponding to the functional $J_{\la,\calP}^{(s)}$ with $\calP$ satisfying
the two-species form \eqref{def:Ptg}.
In view of Proposition~\ref{prop:MTP}--(ii), see also \cite{Suzuki2008book}, 
we have that the optimal Moser-Trudinger constant for $\Js$ is independent of $\tau,\ga$:
\[
\inf_{H_0^1(\Om)}\Js>-\infty\quad\hbox{if and only if }\la\le8\pi.
\]
Our next result shows that, on the contrary, the solution set of \eqref{eq:stoch} for $\la$ 
near $8\pi$ 
significantly depends on the parameters $\tau,\ga$.
Indeed, for fixed $\tau,\ga$ let
\beq
\label{def:lastar}
\lstar:=\sup\{\la>0: \hbox{\eqref{eq:stoch} admits a solution}\}.
\eeq
By direct minimization of $\Js$ and the Moser-Trudinger inequality in \cite{RiZe2012},
we have $\lstar\ge8\pi$. 
If $\Om$ is the unit disc, for small values of $\tau$, the \emph{strict} inequality holds true.
\begin{thm}
\label{thm:Neri}
Suppose $\Om=\BB(0)$.
For every fixed $\ga\in(0,1)$ there exists 
$\underline\tau_\ga\in(0,1)$
such that $\la_{\tau,\ga}^*>8\pi$ for all $\tau\in(0,\underline\tau_\ga)$.
\end{thm}
In particular, Theorem~\ref{thm:det}
and Theorem~\ref{thm:Neri} imply that the bifurcation diagram 
for the deterministic problem \eqref{eq:det}
is qualitatively more similar to the bifurcation diagram for \eqref{eq:std} 
than the corresponding diagram for \eqref{eq:stoch}.
This appears to be the first situation where \eqref{eq:det} is more similar to
\eqref{eq:std} than \eqref{eq:stoch}.
\par
Henceforth we assume $\Om=\BB(0)$.
\subsubsection*{Notation}
We omit the integration variable when it is clear from the context.
For $\sigma>0$ we denote by $B_\sigma$ the disc centered at $0$ with radius $\si$.
We denote by $C>0$ a general constant whose actual value is allowed to vary
from line to line.
\section{Preliminaries}
\label{sec:prelims}
We begin by establishing relations between the non-local mean field problems 
\eqref{eq:det} and \eqref{eq:stoch} defined on $\BB$, namely problem
\begin{equation}
\label{eq:detB}
\begin{cases}
\begin{aligned}
-\Delta v=&\la\left(\tau\frac{e^v}{\int_{\BB}e^v}
+(1-\tau)\ga\frac{e^{\ga v}}{\int_{\BB}e^{\ga v}}\right)
&&\hbox{in\ }\BB\\
v=&0
&&\hbox{on\ }\pl\BB,
\end{aligned}
\end{cases}
\end{equation}
and problem
\begin{equation}
\label{eq:stochB}
\left\{
\bal
-\Delta v=&\la\frac{\tau e^v+(1-\tau)\ga e^{\ga v}}{\int_{\BB}\{\tau e^v+(1-\tau)e^{\ga v}\}}
&&\hbox{in }\BB\\
v=&0&&\hbox{on }\pl\BB,
\eal
\right.
\end{equation}
and the following local problem defined 
on whole space:
\begin{equation}
\label{eq:PT}
\begin{cases}
-\Delta z=e^z+e^{\ga z}
&\hbox{in\ }\rr^2\\
z(0)=\max_{\rr^2}z=\al\in\rr\\
\int_{\rr^2}\{e^z+e^{\ga z}\}=m(\al)<+\infty,
\end{cases}
\end{equation}
which is a special case of the cosmic string problem considered 
in \cite{PoliakovskyTarantello, ChenGuoSpirn}.
We note that solutions to \eqref{eq:detB} and \eqref{eq:stochB}
are radially symmetric \cite{GNN}. 
\par
Indeed, our aim in this section is to establish the following properties.
\begin{prop}
[Reduction of the deterministic problem to \eqref{eq:PT}]
\label{prop:reddet}
Let $v$ be a solution to problem~\eqref{eq:detB}.
Then, the function $z$ defined by
\beq
\label{def:zdet}
z(y)=v(\frac{y}{\si})+\frac{1}{1-\ga}\ln\left[\frac{\tau}{(1-\tau)\ga}
\frac{\int_{\BB}e^{\ga v}}{\int_{\BB}e^v}\right],
\qquad y\in B_\si
\eeq
where $\si$ is defined by
\[
\si^{2(1-\ga)}=\frac{\la^{1-\ga}(1-\tau)\ga}{\tau^\ga}
\frac{(\int_{\BB}e^v)^{\ga}}{\int_{\BB}e^{\ga v}},
\]
is a solution to the problem
\beq
\label{eq:zdet}
\bca
-\Delta z=e^z+e^{\ga z}&\hbox{in }\Bs\\
\int_{\Bs}e^z=\la\tau\\
\int_{\Bs}e^{\ga z}=\la(1-\tau)\ga
\eca
\eeq
satisfying
\[
z(0)=\max_{\rr^2}z=v(0)+\frac{1}{1-\ga}\ln\left[\frac{\tau}{(1-\tau)\ga}
\frac{\int_{\BB}e^{\ga v}}{\int_{\BB}e^v}\right].
\]
\end{prop}
Similarly, we also have:
\begin{prop}
[Reduction of the stochastic problem to \eqref{eq:PT}]
\label{prop:redstoch}
Let $v$ be a solution to problem~\eqref{eq:stochB}.
Then, the function $z$ defined by
\beq
\label{def:zstoch}
z(y)=v(\frac{y}{\si})+\frac{1}{1-\ga}\ln\left[\frac{\tau}{(1-\tau)\ga}\right],
\qquad y\in B_\si
\eeq
where $\si$ is defined by
\[
\si^{2(1-\ga)}=\frac{\la^{1-\ga}(1-\tau)\ga}{\tau^\ga}
\left(\frac{1}{\int_{\BB}\{\tau e^v+(1-\tau)e^{\ga v}\}}\right)^{1-\ga},
\]
is a solution to the problem
\beq
\label{eq:zstoch}
\bca
-\Delta z=e^z+e^{\ga z},
&\hbox{in }\Bs\\
z=\frac{1}{1-\ga}\ln\frac{\tau}{(1-\tau)\ga},
&\hbox{on }\pl\Bs\\
\int_{\Bs}\{e^z+\frac{e^{\ga z}}{\ga}\}=\la
\eca
\eeq
satisfying
\[
z(0)=\max_{\rr^2}z=v(0)+\frac{1}{1-\ga}\ln\left[\frac{\tau}{(1-\tau)\ga}\right].
\]
\end{prop}
Once we are reduced to problem~\eqref{eq:PT},
we may use the following known results.
\begin{lemma}[\cite{Yang1994}]
\label{lem:Yang}
For any $a,b>0$ and for any $\al\in\rr$, the initial value problem:
\beq
\label{eq:Yang}
\bca
-\left(\eta''+\frac{1}{r}\eta'\right)=a e^{\eta}+b e^{\ga\eta},&r>0\\
\eta(0)=\al,\ \eta'(0)=0
\eca
\eeq
admits a unique globally defined solution for any $\al\in\rr$.
Moreover, there exists $\beta>2/\ga$ such that
\[
\lim_{r\to+\infty}\frac{\eta(r)}{\ln r}=-\beta.
\]
\end{lemma}
\begin{proof}
In order to obtain the local existence of solutions to \eqref{eq:Yang},
we rewrite the problem in the form $\eta'=F(r,\eta)$, $\eta(0)=\al$,
where
\[
F(r,\eta)=
\bca
-\frac{1}{r}\int_0^rsf(\eta(s))\,ds,&\hbox{if }r\ge0\\
0&\hbox{if }r<0
\eca
\]
and $f(t)=ae^t+be^{\ga t}$. It is readily checked that $F$ is continuous and locally Lipschitz
continuous
with respect to $\eta$. Hence, standard ODE theory yields the asserted local existence and uniqueness.
Now, the global existence and the asymptotic behavior follow by 
Lemma~5.1 and Lemma~5.2 in \cite{Yang1994}, see also \cite{BerestyckiLionsPeletier1981}.
\end{proof}
The sharp range of admissible values of the constant $\beta$ in Lemma~\ref{lem:Yang}
was determined in \cite{PoliakovskyTarantello}.
\begin{lemma}[\cite{PoliakovskyTarantello, ChenGuoSpirn}]
\label{lem:PT}
For every $\al\in\rr$ there exists a unique solution to \eqref{eq:PT}.
Such a solution is radially symmmetric with respect to $0$ 
and strictly decreasing with respect to $r=|x|$.
Moreover, $m(\al)$ is strictly decreasing with respect to $\al\in\rr$ and it satisfies
\[
\begin{aligned}
&\lim_{\al\to+\infty}m(\al)=8\pi\max\left\{1,\frac{1-\ga}{\ga}\right\}=
\begin{cases}
8\pi&\hbox{if\ }1/2\leq\ga<1,\\
8\pi\frac{1-\ga}{\ga}
&\hbox{if\ }0<\ga<1/2,
\end{cases}\\
&\lim_{\al\to-\infty}m(\al)=\frac{8\pi}{\ga}.
\end{aligned}
\]
\end{lemma}
We set
\beq
\label{def:mm}
m_1(\al):=\int_{\rr^2} e^z
\qquad  m_\ga(\al):=\int_{\rr^2}e^{\ga z},
\eeq
where $z$ is a solution to \eqref{eq:PT}.
We also need an \lq\lq energy identity" from \cite{ChenGuoSpirn}.
\begin{lemma}[\cite{ChenGuoSpirn}]
\label{lem:CGS}
The following relation holds true:
\[
(m_1(\al)+m_\ga(\al))^2=8\pi\left(m_1(\al)+\frac{m_\ga(\al)}{\ga}\right)
\]
for any $\al\in\rr$.
\end{lemma}
\begin{rmk}
From Lemma~\ref{lem:PT} and Lemma~\ref{lem:CGS} it follows in particular that
\beq
\label{eq:mle8pi}
m_1(\al)<8\pi\qquad \hbox{for all } \al\in\rr.
\eeq
\end{rmk}
In order to establish Proposition~\ref{prop:reddet} and Proposition~\ref{prop:redstoch},
we begin by observing the following general fact.
\begin{lemma}
\label{lem:redab}
Let $v$ be a solution to the problem:
\beq
\label{eq:ab}
\bca
\begin{aligned}
-\Delta v=&a\,e^v+b\,e^{\ga v}
&&\hbox{in\ }\BB\\
v=&0
&&\hbox{on\ }\pl\BB,
\end{aligned}
\eca
\eeq
where $a,b>0$.
Let
\begin{equation}
\label{def:z}
z(y):=v(\frac{y}{\sigma})+\ln(\frac{a}{\si^2}),
\qquad y\in\Bs,
\end{equation}
with $\si>0$ defined by
\begin{equation}
\label{def:mu}
\si^{2(1-\ga)}=\frac{b}{a^\ga}.
\end{equation}
Then, $z$ is a solution to the problem:
\beq
\label{eq:z}
\bca
-\Delta z=e^z+e^{\ga z}&\hbox{in }\Bs\\
z=\frac{1}{1-\ga}\ln\frac{a}{b}&\hbox{on }\pl\Bs\\
\int_{\Bs}e^{z(y)}\,dy=\int_{\BB}ae^{v(x)}\,dx\\
\int_{\Bs}e^{\ga z(y)}\,dy=\int_{\BB}be^{\ga v(x)}\,dx.
\eca
\eeq
\end{lemma}
\begin{proof}
The function $w:=v+\ln a$ satisfies
\beq
\bca
-\Delta w=e^w+\frac{b}{a^\ga}e^{\ga w}&\hbox{in }\BB\\
w=\ln a&\hbox{on }\pl\BB\\
\int_{\BB}e^w
=\int_{\BB}ae^v\\
\int_{\BB}\frac{b}{a^\ga}e^{\ga w}
=\int_{\BB}be^{\ga v}.
\eca
\eeq 
For $\si>0$ we rescale
\[
z(y):=w(\frac{y}{\si})+2\ln\frac{1}{\si}=v(\frac{y}{\si})+\ln(\frac{a}{\si^2}),
\qquad y\in\Bs.
\]
Then $z$ satisfies
\[
\bca
-\Delta z=e^z+\si^{-2(1-\ga)}\frac{b}{a^\ga}e^{\ga z}&\hbox{in }\Bs\\
z=\ln\frac{a}{\si^2}&\hbox{on }\pl\Bs\\
\int_{\Bs}e^{z(y)}\,dy=\int_{\BB}e^{w(x)}\,dx\\
\int_{\Bs}\si^{-2(1-\ga)}e^{\ga z(y)}\,dy=\int_{\BB}e^{\ga w(x)}\,dx.
\eca
\]
We observe that $\eqref{def:mu}$ implies that
\beq
\label{eq:abs}
\frac{a}{\si^2}=\left(\frac{a}{b}\right)^{\frac{1}{1-\ga}}.
\eeq
We deduce that if $\si$ is defined by $\eqref{def:mu}$,
then $z$ satisfies \eqref{eq:z}, as asserted.
\end{proof}
\begin{proof}
[Proof of Proposition~\ref{prop:reddet}]
We apply Lemma~\ref{lem:redab} with
\[
a=\frac{\la\tau}{\int_{\BB}e^v},
\qquad\qquad
b=\frac{\la(1-\tau)\ga}{\int_{\BB}e^{\ga v}}.
\]
Choosing $\si$ according to \eqref{def:mu}:
\[
\si^{2(1-\ga)}=\frac{b}{a^\ga}=\frac{\la^{1-\ga}(1-\tau)\ga}{\tau^\ga}
\frac{\left(\int_{\BB}e^v\right)^\ga}{\int_{\BB}e^{\ga v}},
\]
we obtain, recalling \eqref{eq:abs}, that
\[
\frac{a}{\si^2}=\left[\frac{\tau}{(1-\tau)\ga}
\frac{\int_{\BB}e^{\ga v}}{\int_{\BB}e^v}\right]^{\frac{1}{1-\ga}}.
\]
Thus, the function~$z$ defined by \eqref{def:zdet} satisfies the problem
\[
\bca
-\Delta z=e^z+e^{\ga z}&\hbox{in }\Bs\\
z=\frac{1}{1-\ga}\ln[\frac{\tau}{(1-\tau)\ga}
\frac{\int_{\BB}e^{\ga v}}{\int_{\BB}e^v}]
&\hbox{on }\pl\Bs\\
\int_{\Bs}e^z=\int_{\BB}ae^v=\la\tau\\
\int_{\Bs}e^{\ga z}=\int_{\BB}be^{\ga v}=\la(1-\tau)\ga
\eca
\]
and the asserted reduction is established.
\end{proof}
Similarly, we establish the reduction of problem~\eqref{eq:stochB}.
\begin{proof}
[Proof of Proposition~\ref{prop:redstoch}]
We apply Lemma~\ref{lem:redab} with
\[
a=\frac{\la\tau}{\int_{\BB}\{\tau e^v+(1-\tau)e^{\ga v}\}},
\qquad\qquad
b=\frac{\la(1-\tau)\ga}{\int_{\BB}\{\tau e^v+(1-\tau)e^{\ga v}\}}.
\]
Choosing $\si>0$ according to \eqref{def:mu}:
\[
\si^{2(1-\ga)}=\frac{b}{a^\ga}
=\frac{\la^{1-\ga}(1-\tau)\tau^{-\ga}\ga}{(\int_{\BB}\{\tau e^v+(1-\tau)e^{\ga v}\})^{1-\ga}},
\]
we find, recalling \eqref{eq:abs}, that
\[
\frac{a}{\si^2}=\left[\frac{\tau}{(1-\tau)\ga}\right]^{\frac{1}{1-\ga}}.
\]
Moreover,
\[
\bal
\int_{\Bs}e^z=&\int_{\BB}ae^v
=\frac{\la\tau\int_{\BB}e^v}{\int_{\BB}\{\tau e^v+(1-\tau)e^{\ga v}\}}\\
\int_{\Bs}e^{\ga z}
=&\int_{\BB}be^{\ga v}=\frac{\la(1-\tau)\ga\int_{\BB}e^{\ga v}}{\int_{\BB}\{\tau e^v+(1-\tau)e^{\ga v}\}}.
\eal
\]
Then, the function~$z$ 
defined by \eqref{def:zstoch} is a solution to the problem:
\[
\bca
-\Delta z=e^z+e^{\ga z},
&\hbox{in }\Bs\\
z=\frac{1}{1-\ga}\ln\frac{\tau}{(1-\tau)\ga},
&\hbox{on }\pl\Bs\\
\int_{\Bs}e^z+\frac{e^{\ga z}}{\ga}=\la,
\eca
\]
as asserted.
\end{proof}
\section{Proof of Theorem~\ref{thm:det}}
\label{sec:det}
In order to prove Theorem~\ref{thm:det} we use the following Pohozaev identity.
\begin{lemma}[Pohozaev identity]
\label{lem:Pohozaev}
The following identity holds true:
\[
-\frac{1}{2}\int_{\partial\BB}|\nabla v|^2\,d\sigma
=-2\la+2\pi\la\left(\frac{\tau}{\int_{\BB} e^v}+\frac{1-\tau}{\int_{\BB} e^{\ga v}}\right)
\]
\end{lemma}
\begin{proof}
Multiplying \eqref{eq:det} by $x\cdot\nabla v$ and integrating we have
\[
\int_{\BB}(x\cdot\nabla v)(-\Delta v)\,dx=-\frac{1}{2}\int_{\partial\BB}|\nabla v|^2\,d\sigma.
\]
Similarly,
\[
\begin{aligned}
\int_{\BB}(x\cdot\nabla v)\la&\left(\tau\frac{e^v}{\int_{\BB}e^v}
+(1-\tau)\ga\frac{e^{\ga v}}{\int_{\BB}e^{\ga v}}\right)
=\la\int_{\BB} x\cdot\left(\tau\frac{\nabla e^v}{\int_{\BB}e^v}
+(1-\tau)\frac{\nabla e^{\ga v}}{\int_{\BB}e^{\ga v}}\right)\\
=&\la\int_{\pl\BB}\left(\frac{\tau}{\int_{\BB} e^v}+\frac{1-\tau}{\int_{\BB} e^{\ga v}}\right)-2\la
=-2\la+2\pi\la\left(\frac{\tau}{\int_{\BB} e^v}+\frac{1-\tau}{\int_{\BB} e^{\ga v}}\right).
\end{aligned}
\]
\end{proof}
\begin{proof}[Proof of Theorem~\ref{thm:det}]
We prove separately the two cases, corresponding to the
two possible values of the Moser-Trudinger constant $\bl$ defined in \eqref{eq:MTdet}. 
Throughout this proof, for the sake of simplicity, we denote $\bar\la=\bl$. 
\par
\textit{Case~1: $\bar\la=\frac{8\pi}{\{\tau+(1-\tau)\ga\}^2}$}.
\par
Recalling that $\Om=\BB$, and using the Schwarz inequality, we have
\[
\begin{aligned}
\int_{\pl\BB}|\nabla v|^2\,d\sigma=&\int_{\pl\BB}|\frac{\pl v}{\pl\nu}|^2\,d\sigma
\ge\frac{1}{|\pl\BB|}\left(\int_{\pl\BB}(-\frac{\pl v}{\pl\nu})\,d\sigma\right)^2\\
=&\frac{1}{2\pi}\left(\int_{\BB}(-\Delta v)\,dx\right)^2
=\frac{\la^2}{2\pi}\left\{\tau+(1-\tau)\ga\right\}^2.
\end{aligned}
\]
In view of Lemma~\ref{lem:Pohozaev} we deduce that
\[
\la\le\frac{8\pi}{\{\tau+(1-\tau)\ga\}^2}\left\{1-\pi\left(\frac{\tau}{\int_{\BB} e^v}
+\frac{1-\tau}{\int_{\BB} e^{\ga v}}\right)\right\}<\frac{8\pi}{\{\tau+(1-\tau)\ga\}^2}
=\bar\la.
\]
\par
\textit{Case~2:\ }$\bar\la=\frac{8\pi}{\tau}$.
\par
Let $z$ be defined as in Proposition~\ref{prop:reddet}. 
By Lemma \ref{lem:Yang} and Lemma \ref{lem:PT}, 
$z$ coincides in $B_\si$ with the unique solution $z_\al$
to \eqref{eq:PT} with $\al=\al^*$ given by
\[
\al^*=v(0)+\frac{1}{1-\ga}\ln\frac{\tau\int_{\BB}e^{\ga v}}{(1-\tau)\ga\int_{\BB}e^v}.
\]
Consequently, \eqref{eq:mle8pi} assures 
\[
\int_{B_\si}e^z
=\la\tau<\int_{\rr^2}e^{z_{\al^*}}=m_1(\al^*)<8\pi.
\]
In particular, we derive from the above that
\[
\la<\frac{8\pi}{\tau}=\bar\la.
\]
Now, Theorem~\ref{thm:det} is completely established.
\end{proof}
\section{Proof of Theorem~\ref{thm:Neri}}
\label{sec:stoch}
Our first aim in this section is to establish the following \lq\lq converse" for
Proposition~\ref{prop:redstoch}.
It is convenient to define
\beq
\label{def:btg}
\btg:=\frac{1}{1-\ga}\ln\frac{\tau}{(1-\tau)\ga}.
\eeq
\begin{prop}
\label{prop:redconv}
Let $z$ be a solution to \eqref{eq:PT} with
$\al>\btg$.
Let $\si>0$ be defined by
\beq
\label{def:si}
z\vert_{\pl B_\si}=\btg
\eeq
and let $a>0$ be defined by
\beq
\label{def:a}
\frac{a}{\si^2}=e^{\btg}=\left[\frac{\tau}{(1-\tau)\ga}\right]^{\frac{1}{1-\ga}}.
\eeq
Then, the function $v$ defined by
\beq
\label{def:vred}
v(x)=z(\si x)+\ln\frac{\si^2}{a}=z(\si x)-\btg
\eeq
is a solution to problem~\eqref{eq:stochB} with
\beq
\label{def:laz}
\lambda=\int_{B_\si}\left(e^z+\frac{e^{\ga z}}{\ga}\right).
\eeq
\end{prop}
\begin{proof}
The function $v$ defined in \eqref{def:vred} satisfies the equation:
\beq
\label{eq:vred}
-\Delta v=a e^v+\si^{2(1-\ga)}a^\ga e^{\ga v}
=\si^2\left[\frac{a}{\si^2}\,e^v+\left(\frac{a}{\si^2}\right)^\ga e^{\ga v}\right].
\eeq
Hence, we obtain from \eqref{eq:stochB} the necessary condition
\[
\frac{a}{\si^2\la\tau}=\left(\frac{a}{\si^2}\right)^\ga\frac{1}{\la(1-\tau)\ga}
=\frac{1}{\si^2\int_{\BB}\{\tau e^v+(1-\tau)e^{\ga v}\}},
\]
from which \eqref{def:a} follows.
On the other hand,  \eqref{def:si} is a consequence of the Dirichlet boundary condition
for $v$.
\par
Finally, we compute, for $v$ give by \eqref{def:vred}:
\[
\bal
\int_{\BB}e^v\,dx=&
\frac{\si^2}{a}\int_{\BB}e^{z(\si x)}\,dx=\frac{e^{-\btg}}{\si^2}\int_{\Bs}e^z,\\
\int_{\BB}e^{\ga v}
=&\left(\frac{\si^2}{a}\right)^\ga\int_{\BB}e^{\ga z(\si x)}\,dx=\frac{e^{-\ga\btg}}{\si^2}\int_{\Bs}e^{\ga z}.
\eal
\]
In view of the elementary identity:
\beq
\label{eq:elemid}
\tau e^{-\btg}=\left[\frac{(1-\tau)\ga}{\tau^\ga}\right]^{\frac{1}{1-\ga}}
=(1-\tau)\ga e^{-\ga\btg},
\eeq
we may write
\beq
\label{eq:intid}
\int_{\BB}\{\tau e^v+(1-\tau)e^{\ga v}\}
=\si^{-2}\left[\frac{(1-\tau)\ga}{\tau^\ga}\right]^{\frac{1}{1-\ga}}
\int_{\Bs}(e^z+\frac{e^{\ga z}}{\ga}).
\eeq
Inserting \eqref{eq:elemid}--\eqref{eq:intid} into \eqref{eq:vred} we deduce that
\beq
\bal
-\Delta v=&\si^2\left[e^{\btg}\,e^v+e^{\ga\btg} e^{\ga v}\right]
=\si^2\left[\frac{\tau^\ga}{(1-\tau)\ga}\right]^{\frac{1}{1-\ga}}
\left\{\tau e^v+(1-\tau)\ga e^{\ga v}\right\}\\
=&\frac{\tau e^v+(1-\tau)\ga e^{\ga v}}{\int_{\BB}\{\tau e^v+(1-\tau)e^{\ga v}\}}
\,\int_{\Bs}(e^z+\frac{e^{\ga z}}{\ga}).
\eal
\eeq
We deduce that $v$ defined by \eqref{def:vred} with \eqref{def:si}--\eqref{def:a} 
is a solution to problem~\eqref{eq:stochB} with
\[
\la=\int_{\Bs}(e^z+\frac{e^{\ga z}}{\ga}),
\]
as asserted.
\end{proof}
In view of Proposition~\ref{prop:redconv} it is natural to consider, for fixed $\tau,\ga\in(0,1)$
and $\al>\btg$, where $\btg$ is defined in \eqref{def:btg},
solution pairs $(z,\si)=(z_\ga(\al),\si_{\tau,\ga}(\al))$ 
where $z_\ga(\al)$ is a solution to problem~\eqref{eq:PT}
and $\si_{\tau,\ga}(\al)$ is uniquely defined by \eqref{def:si}.
Correspondingly, we obtain $\Lambda_{\tau,\ga}(\al)$ defined by
\beq
\label{def:Latg}
\Lambda_{\tau,\ga}(\al)=\int_{B_{\si_{\tau,\ga}(\al)}}\left\{e^{z_{\ga}(\al)}
+\frac{e^{\ga z_{\ga}(\al)}}{\ga}\right\}.
\eeq
Let
\beq
\label{def:II}
\bal
\mathcal I_{\tau,\ga}:=&\{\la>0:\ \hbox{there exists a solution to \eqref{eq:stochB}}\}\\
\mathcal I_{\tau,\ga}':=&\left\{\Lambda>0:\ \Lambda=\Lambda_{\tau,\ga}(\al)\ 
\hbox{for some }\al>\btg\right\}.
\eal
\eeq
In view of Proposition \ref{prop:redstoch} and Proposition~\ref{prop:redconv}, we have
\beq
\label{eq:II}
\mathcal I_{\tau,\ga}=\mathcal I_{\tau,\ga}'.
\eeq
Therefore, we are reduced to the study of $\mathcal I_{\tau,\ga}'$.
Setting
\[
\Lambda_{\tau,\ga}^*:=\sup\mathcal I_{\tau,\ga}',
\]
we have
\[
\lstar=\Lambda_{\tau,\ga}^*\ge8\pi,
\]
where $\lstar$ is defined in \eqref{def:lastar}.
\begin{prop}
\label{prop:siLa}
Let $\tau,\ga\in(0,1)$ be fixed.
The following properties hold true:
\begin{itemize}
\item[(i)]
The quantities $\sitg(\al)$, $\Latg(\al)$
are continuous functions of $\al\in[\btg,+\infty)$,
provided that we set $\sitg(\btg)=\Latg(\btg)=0$.
\item[(ii)]
There holds
\[
\lim_{\al\to+\infty}\sitg(\al)=0.
\] 
\item[(iii)]
There holds
\[
\lim_{\al\to+\infty}\Latg(\al)=8\pi.
\]
\end{itemize}
\end{prop}
\begin{proof}
Throughout this proof, for simplicity we omit the subscripts $\tau,\ga$.
\par
Proof of (i). 
Given $\al\in(\btg,+\infty)$, $\si(\al)$ and $\Lambda(\al)$ 
are uniquely determined by $\btg$ and $z(\alpha)$, 
the unique solution to \eqref{eq:PT}, only. 
Moreover, it is also the unique solution to \eqref{eq:Yang} for $a=b=1$. 
Therefore, the desired continuity follows from the standard ODE theory, since $\btg$ is independent of $\alpha$. 
\par
Proof of (ii).
Let $\al_k\to+\infty$. Let $(z_k,\si_k)$ be the solution pair to
\eqref{eq:PT} and \eqref{def:si} with $\al=\al_k$
and $\Lambda_k$ defined by \eqref{def:Latg}.
We claim that $\lim_{k\to+\infty}\si_k=0$.
Indeed, suppose that there exists $\si_0>0$ and a subsequence,
still denoted $\si_k$, such that $\lim_{k\to+\infty}\si_k=\si_0>0$.
We use an argument from \cite{BrezisMerle}.
Let $\uz_k$ be defined as the unique solution to the problem
\[
\bca
-\Delta\uz_k=e^{z_k}+e^{\ga z_k}&\hbox{in\ }B_{\si_0/2}\\
\uz_k=z_k|_{\pl B_{\si_k}}=\btg
&\hbox{on\ }\pl B_{\si_0/2}.
\eca
\]
Then, for $k$ sufficiently large, we have $z_k\ge\uz_k\ge\btg$ in $B_{\si_0/2}$ 
in view of the maximum principle.
On the other hand, setting $V_k(x):=1+e^{-(1-\ga)z_k}$ we may write
\[
\bca
-\Delta z_k=e^{z_k}+e^{\ga z_k}=V_ke^{z_k}&\hbox{in }B_{\si_k}\\
\int_{B_{\si_k}}V_ke^{z_k}\le\int_{B_{\si_k}}\{e^{z_k}+\frac{e^{\ga z_k}}{\ga}\},
\eca
\]
with 
\[
\bal
&\|V_k\|_{L^\infty(B_{\si_0})}\le 1+e^{-(1-\ga)\btg},\\
&z_k(0)=\max_{B_{\si_0}}z_k\to+\infty,\\
&\int_{B_{\si_k}}e^{z_k}\le m_1(\al_k)\le8\pi,
\qquad\ \int_{B_{\si_k}}e^{\ga z_k}\le m_\ga(\al_k)\le\frac{8\pi}{\ga},
\eal
\]
where $m_1(\al_k)$, $m_\ga(\al_k)$ are the masses defined in \eqref{def:mm}.
In view of an adaptation of the Brezis-Merle alternative,
as may be found, e.g., in \cite{ORS,DeMarchisRicciardi}, we conclude that
there exist $n\ge4\pi$ and $s\ge0$, $s\in L^1(B_{\si_0})$
such that, up to subsequences,
\[
e^{z_k}+e^{\ga z_k}=V_ke^{z_k}\stackrel{\ast}{\rightharpoonup}n\de_0(dx)+s(x)dx.
\]
We deduce that 
\[
\uz_k(x)\ge\frac{n}{2\pi}G_{B_{\si_0/2}}(x,0)+\btg
\]
where $G_{B_{\si_0/2}}(\cdot,\cdot)$ denotes the Green's function on $B_{\si_0/2}$.
Since $n/(2\pi)\ge2$, Fatou's lemma implies that
\[
\int_{B_{\si_0/2}}e^{z_k}\ge\int_{B_{\si_0/2}}e^{\uz_k}\to+\infty
\]
as $k\to+\infty$, a contradiction.
\par
Proof of (iii).
In view of Proposition~\ref{prop:redconv}, setting
\[
\bal
&\la_k=\int_{B_{\si_k}}\{e^{z_k}+\frac{e^{\ga z_k}}{\ga}\}=\Lambda_k,\\
&v_k(x)=z_k(\si_kx)-\btg,
\eal
\]
we obtain a solution $v_k$ to problem~\eqref{eq:stochB} with $\la=\la_k$.
Moreover, we have
\[
v_k(0)=\al_k-\btg\to+\infty.
\]
Therefore $(\la_k,v_k)$ is a blowing up solution sequence for 
\eqref{eq:stochB} with a unique blow-up point located at $0$. 
Now, in view of the mass quantization results as established 
in \cite{RiZe2016, RicciardiTakahashi},
we conclude that $\lim_{k\to+\infty}\la_k=\lim_{k\to+\infty}\Lambda_k=8\pi$.
\end{proof}
Now, we are ready to prove our main result for \eqref{eq:stochB}, 
using Proposition \ref{prop:siLa} and Lemma \ref{lem:PT}.
\begin{proof}[Proof of Theorem~\ref{thm:Neri}]
We fix $\ga\in(0,1)$.
We observe that $\btg$ as defined in \eqref{def:btg} is monotonically increasing
with respect to $\tau\in(0,1)$ and $\lim_{\tau\to0^+}\btg=-\infty$.
Consequently, for any $\tau_1,\tau_2\in(0,1)$ with $\tau_1<\tau_2$
and for any fixed $\al\in(\beta_{\tau_2,\ga},+\infty)$,
we have $\si_{\tau_1,\ga}(\al)>\si_{\tau_2,\ga}(\al)$
and consequently $\Lambda_{\tau_1,\ga}(\al)>\Lambda_{\tau_2,\ga}(\al)>0$.
Moreover, $\lim_{\tau\to0^+}\sitg(\al)=+\infty$.
\par
Let $\al_0\in\rr$ and $z_0=z_0(x)$ be uniquely defined by
the problem
\[
\bca
-\Delta z_0=e^{z_0}+e^{\ga z_0},&\hbox{in }\rr^2\\
z_0(0)=\max_{\rr^2}z_0=\al_0\\
\int_{\rr^2}\{e^{z_0}+e^{\ga z_0}\}=\frac{8\pi}{2}(1+\frac{1}{\ga})>8\pi.
\eca
\]
We deduce that $\lim_{\tau\to0^+}\Lambda_{\tau,\ga}(\al_0)\ge8\pi(1+\ga^{-1})/2$.
In particular, there exists $0<\tau_0=\tau_0(\ga)\ll1$ such that for all $\tau\in(0,\tau_0)$
we have $\Lambda_{\tau,\ga}(\al_0)>8\pi$.
\end{proof}
\section{Appendix: remarks on $\bl$}
Our aim in this Appendix is to show that the values of $\bl$ given in 
\eqref{eq:MTdet} may be interpreted in terms of Liouville bubbles.
Indeed, for $\de>0$ let
\[
\Ud(x):=\ln\frac{8\de^2}{(\de^2+|x|^2)^2},
\quad x\in\rr^2
\]
be the family of \lq\lq Liouville bubbles", uniquely determined as solutions to the problem
\[
\bca
-\Delta U=e^u&\hbox{in }\rr^2\\
\int_{\rr^2}e^U<+\infty.
\eca
\]
Let $P\Ud$ be the projection of $\Ud$ on $H_0^1(\BB)$, namely
\[
P\Ud(x)=\ln\frac{(\de^2+1)^2}{(|x|^2+1)^2}.
\]
\begin{prop}
\label{prop:Jblowdown}
For every $\ga\in(0,1)$ there exists some $t_\ga>0$ such that for all 
$\la>\bl$ there holds
\[
J_\la(t_\ga P\Ud)\to-\infty
\]
as $\de\to0^+$.
\end{prop}
The following asymptotic behaviors are readily verified.
\begin{lemma}
\label{lem:PUdexp}
The following expansions hold, as $\de\to0^+$:
\[
\bal
\int_{\BB}|\nabla P\Ud(x)|^2\,dx=&16\pi\ln\frac{1}{\de^2}+O(1)\\
\int_{\BB}e^{aP\Ud(x)}\,dx=&
\bca
O(1),&\hbox{if }0<a<\frac{1}{2}\\
\ln\ln\frac{1}{\de^2}+O(1),
&\hbox{if }a=\frac{1}{2}\\
(2a-1)\ln\frac{1}{\de^2}+O(1),
&\hbox{if }\frac{1}{2}<a<1,
\eca
\eal
\]
for any $a\in(0,1)$.
\end{lemma}
For later convenience, we note that
\beq
\label{eq:tau}
0<\frac{\tau}{1+\tau}<\frac{\sqrt\tau}{1+\sqrt\tau}<\frac{1}{2},
\qquad\hbox{for all }0<\tau<1.
\eeq
\begin{lemma}
\label{lem:gataufirst}
Suppose $\ga>\tau/(1+\tau)$. Then, there exists a solution $t=t_\ga$ to the 
system:
\beq
\label{eq:tgasys}
\bca
t\ga>\frac{1}{2}\\
8\pi t^2-2\la(\tau+(1-\tau)\ga)t+\la<0
\eca
\eeq
provided that
\beq
\label{eq:lagebl}
\la>\frac{8\pi}{(\tau+(1-\tau)\ga)^2}.
\eeq
\end{lemma}
\begin{proof}
By elementary considerations, a necessary condition for a solution $t$ to the second inequality in \eqref{eq:tgasys}
to exist is that
\[
\la^2(\tau+(1-\tau)\ga)^2-8\pi\la>0,
\]
and moreover $t_-<t<t_+$ where
\[
t_\pm=\frac{\la(\tau+(1-\tau)\ga)\pm\sqrt{\la^2(\tau+(1-\tau)\ga)^2-8\pi\la}}{8\pi}
\]
hence we readily derive \eqref{eq:lagebl}.
Hence, assuming that \eqref{eq:lagebl} holds true, 
we are left to check compatibility with the first inequality in \eqref{eq:tgasys}.
To this end, we observe that in view of \eqref{eq:lagebl} we have
\[
t_+=\frac{\la(\tau+(1-\tau)\ga)+\sqrt{\la^2(\tau+(1-\tau)\ga)^2-8\pi\la}}{8\pi}
>\frac{1}{\tau+(1-\tau)\ga}.
\]
In particular, we have that a sufficient condition for $t_+\ga>1/2$ 
is given by $\ga/(\tau+(1-\tau)\ga)>1/2$.
The latter inequality is equivalent to $\ga>\tau/(1+\tau)$, which holds true by assumption.
We conclude that there exists a solution $t_\ga$ to \eqref{eq:tgasys}
satisfying $0<t_+-t_\ga\ll1$.
\end{proof}
\begin{lemma}
\label{lem:gatausecond}
If $0<\ga<1/2$, there exists a solution $t_\ga$ to the system
\beq
\label{eq:tgasys2}
\bca
\frac{1}{2}<t<\frac{1}{2\ga}\\
8\pi t^2-\la\tau(2t-1)<0
\eca
\eeq
provided $\la>8\pi/\tau$.
\end{lemma}
\begin{proof}
By elementary considerations, solutions to the second inequality in \eqref{eq:tgasys2} exist
if and only if
\[
(\la\tau)^2-8\pi\la\tau>0,
\]
which is equivalent to $\la>8\pi/\tau$.
Moreover, solutions satisfy $t_-<t<t_+$ where
\[
t_\pm=\frac{\la\tau\pm\sqrt{(\la\tau)^2-8\pi\la\tau}}{8\pi}.
\]
Thus, a sufficient condition for a solution to \eqref{eq:tgasys2} to exist
is given by $1/2<t_-<1/(2\ga)$.
It is readily checked that the inequality $t_->1/2$ is always satisfied,
since we may write
\beq
\begin{aligned}
t_-=\frac{(\la\tau)^2-[(\la\tau)^2-8\pi\la\tau]}{8\pi(\la\tau+\sqrt{(\la\tau)^2-8\pi\la\tau})}
>\frac{1}{2}.
\end{aligned}
\eeq
On the other hand, the inequality $t_-<1/(2\ga)$ is equivalent to
\[
\sqrt{(\la\tau)^2-8\pi\la\tau}>\la\tau-\frac{4\pi}{\ga}.
\]
Therefore, if $\la\tau-4\pi/\ga<0$, the second inequality holds true.
Assuming $\la\tau-4\pi/\ga\ge0$, i.e., $\la>4\pi/(\tau\ga)$, we obtain from above the inequality
\[
\ga^2[(\la\tau)^2-8\pi\la\tau]>(\la\tau\ga-4\pi)^2.
\]
We deduce the equivalent condition
$\la\tau\ga(1-\ga)>2\pi$.
Since $\la>4\pi/(\tau\ga)$, a sufficient condition for the above is
$4\pi(1-\ga)>2\pi$, equivalently $\ga<1/2$.
\end{proof}
Now we can prove Proposition~\ref{prop:Jblowdown}.
\begin{proof}[Proof of Proposition~\ref{prop:Jblowdown}]
We first assume that $\ga\ge\sqrt\tau/(1+\sqrt\tau)$ and therefore 
$\bl=8\pi/(\la\tau\ga(1-\ga))^2$.
In view of \eqref{eq:tau}, we have $\sqrt\tau/(1+\sqrt\tau)>\tau/(1+\tau)$, 
so in particular
$\ga>\tau/(1+\tau)$. We deduce from Lemma~\ref{lem:gataufirst}
that there exists a solution $t_\ga$ to system \eqref{eq:tgasys}.
Now, Lemma~\ref{lem:PUdexp} implies that
\[
\bal
J_\la(t_\ga P\Ud)
=&\left\{8\pi t_\ga^2-\la[\tau(2t_\ga-1)+(1-\tau)(2t_\ga\ga-1)]\right\}\ln\frac{1}{\de^2}
+O(1)\\
=&\left\{8\pi t_\ga^2-2\la(\tau+(1-\tau)\ga)t_\ga+\la\right\}\ln\frac{1}{\de^2}
+O(1)\to-\infty
\eal
\]
as $\de\to0^+$.
Hence, the desired asymptotics is established in this case.
\par
Now, we assume $\ga<\sqrt\tau/(1+\sqrt\tau)$. Then, $\bl=8\pi/\tau$.
Since $\sqrt\tau/(1+\sqrt\tau)<1/2$, we may apply Lemma~\ref{lem:gatausecond}
to conclude that there exists a solution $t_\ga$ to system~\eqref{eq:tgasys2}.
Lemma~\ref{lem:PUdexp} yields
\[
J_\la(t_\ga P\Ud)=\{8\pi t_\ga^2-\la\tau(2t_\ga-1)\}\ln\frac{1}{\de^2}+O(1)\to-\infty
\]
as $\de\to0^+$.
Now Proposition~\ref{prop:Jblowdown} is completely established.
\end{proof}

\section*{Acknowledgments}
The second author thanks Universit\`{a} di Napoli Federico II, where part of this work was completed,
for support and hospitality. 
He was also supported by JSPS KAKENHI Grant Number JP16K17627. 

\end{document}